\documentclass[a4paper,11pt]{article}
\usepackage[T1]{fontenc}
\usepackage[utf8]{inputenc}
\usepackage[english]{babel}
\usepackage[margin=1.1in]{geometry}
\usepackage{microtype}
\usepackage{braket}
\usepackage{amsmath}
\usepackage{amssymb}
\usepackage{amsthm, aliascnt}
\usepackage{mathtools}
\usepackage{mathrsfs}
\usepackage{xcolor}
\usepackage[hyphens]{url}
\usepackage{hyperref}
\usepackage[hyphenbreaks]{breakurl}
\usepackage{enumitem}
\usepackage{tikz}
\usepackage[maxbibnames=99]{biblatex}
\usepackage{csquotes}
\usetikzlibrary{patterns}
\usepackage{tikz}
\usepackage{esint}

\hypersetup{
                colorlinks=true,
                linkcolor={magenta},
                citecolor={cyan},
                breaklinks=true}

\theoremstyle{plain}
\newtheorem{teor}{Theorem}
\numberwithin{teor}{section}
\numberwithin{equation}{section}
\theoremstyle{definition}
\newaliascnt{defi}{teor}
\newtheorem{defi}[defi]{Definition}
\aliascntresetthe{defi}

\theoremstyle{plain}
\newaliascnt{lemma}{teor}%

\aliascntresetthe{lemma}
\theoremstyle{plain}
\newaliascnt{prop}{teor}%
\newtheorem{prop}[prop]{Proposition}
\aliascntresetthe{prop}
\theoremstyle{plain}
\newaliascnt{cor}{teor}%

\aliascntresetthe{cor}
\theoremstyle{definition}
\newaliascnt{ex}{teor}%

\aliascntresetthe{ex}
\theoremstyle{definition}
\newaliascnt{oss}{teor}%
\newtheorem{oss}[oss]{Remark}
\aliascntresetthe{oss}
\theoremstyle{plain}
\newtheorem{open}{Open problem}
\DeclarePairedDelimiter{\abs}{\lvert}{\rvert}

\newcommand{\R}{\mathbb{R}}

\newcommand{\Ln}{\mathcal{L}^n}
\newcommand{\Hn}{\mathcal{H}^{n-1}}

\newcommand{\eps}{\varepsilon}

\DeclareMathOperator{\divv}{div}

\makeatletter
\newcommand{\leqnomode}{\tagsleft@true\let\veqno\@@leqno}
\newcommand{\reqnomode}{\tagsleft@false\let\veqno\@@eqno}

\makeatother

\title{On the optimal shape of a thin insulating layer}
\author{P. Acampora, E. Cristoforoni, C. Nitsch, C. Trombetti}
\date{}
\newcommand{\Addresses}{{%
 \bigskip 
 \footnotesize 
 
 \textsc{Dipartimento di Matematica e Applicazioni ``R. Caccioppoli'', Universit\`a degli studi di Napoli Federico II, Via Cintia, Complesso Universitario Monte S. Angelo, 80126 Napoli, Italy.}\par\nopagebreak 
 
 \medskip 
 
 \textit{E-mail address}, P.~Acampora: \texttt{paolo.acampora@unina.it} 
  
 \medskip 
 
 \textit{E-mail address}, C.~Nitsch: \texttt{c.nitsch@unina.it}

  \medskip 
 
 \textit{E-mail address}, C.~Trombetti: \texttt{cristina@unina.it} 
 
 \medskip 
 
\textsc{Mathematical and Physical Sciences for Advanced Materials and Technologies, Scuola Superiore Meridionale, Largo San Marcellino 10, 80126, Napoli, Italy.}\par\nopagebreak 
 
 \medskip 
 
 \textit{E-mail address}, E.~Cristoforoni: \texttt{emanuele.cristoforoni@unina.it} 
}}

\begin{document}
\reversemarginpar
\maketitle
\begin{abstract}
    We are interested in the thermal insulation of a bounded open set $\Omega$ surrounded by a set whose thickness is locally described by $\eps h$, where $h$ is a non-negative function defined on the boundary $\partial\Omega$. We study the problem in the limit for $\eps$ going to zero using a first-order asymptotic development by $\Gamma$-convergence. 
    

\end{abstract}
\renewcommand*{\sectionautorefname}{Section}
\renewcommand*{\subsectionautorefname}{Subsection}
\section{Introduction}
Energy efficiency has emerged as one of the most pressing issues in recent years, as it is critical to achieving sustainable development, lowering greenhouse gas emissions, and mitigating climate change, all while boosting economic growth and increasing quality of life.
Thermal insulation is important in the context of energy efficiency because it helps to reduce heat transfer and energy losses in buildings and industrial processes, resulting in significant energy and cost savings.

This paper addresses the topic of thermal insulation of a solid, kept at constant temperature, by displacing around it an insulator. Roughly speaking, the optimal insulation is the one minimizing, at thermal equilibrium, the heat rate loss per unit time across the exterior boundary. More specifically, in our model heat exchange with the environment occurs through convection, which is by far the most common mechanism in real world applications. 

In order to contain costs of insulation, the volume of the insulator is prescribed. The resulting mathematical model gives rise to a free boundary problem which has been previously studied in \cite{CK} and for more general heat transfer mechanisms (including for instance also radiation) in \cite{BNNT}. In particular, the former paved the way on how to prove the existence of an optimal distribution of insulator and the regularity of its a priori unknown boundary. %

However, although a solution always exists for every set to be insulated and every amount of insulator, very little is known about the optimal shape of the insulator, and the only completely solved case happens to be the radial one.

This work addresses the problem of qualitatively and quantitatively describing how to displace the insulator, and to do so, it restricts the analysis to "small insulation thickness", an approximation reasonable in many contexts, such as in the case of buildings and other big structures. %

Throughout the paper, $\Omega$ represents the body to insulate, in which the temperature is fixed. Since the problem is invariant under temperature scaling and translations, without loss of generality the temperature of the body will be $1$, while the environment temperature will be $0$.
If convection is the leading mechanism of thermal exchange with the environment, then the heat rate loss per unit time and unit surface is proportional to the temperature jump across the surface element separating the insulator (or the body) with the environment. The constant of proportionality will be denoted by $\beta>0$. If $\Sigma$ denotes the insulator (see figure), then at equilibrium the temperature distribution in $\Sigma$ is an harmonic function which we denote by $u$. 
\begin{figure}
\caption{Body $\Omega$ with a thin insulating layer $\Sigma$}
\centering
\begin{tikzpicture}
    \node[anchor=south west,inner sep=0,opacity=0.6] (image) at (0,0) {\includegraphics[width=1\textwidth]{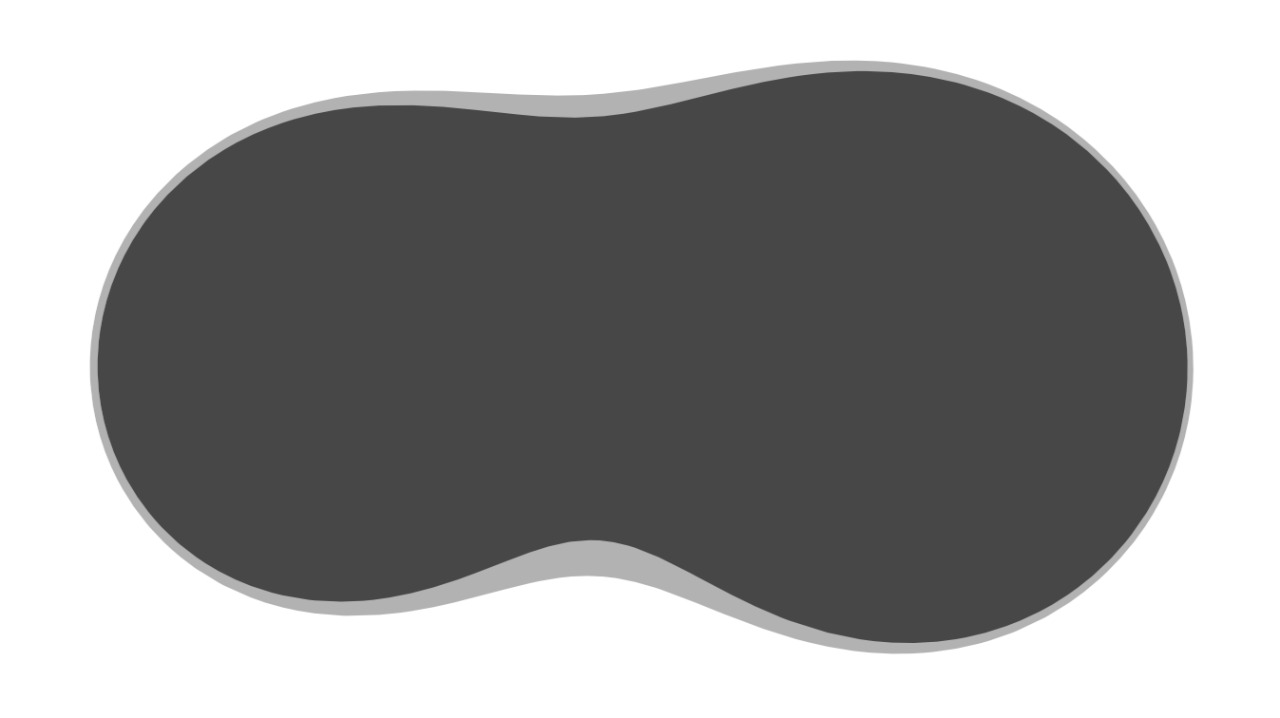}};
    \begin{scope}[x={(image.south east)},y={(image.north west)}]
        \node[anchor=south west] at (0.5, 0.5) {\Huge$\Omega$};
        \node[anchor=south west] at (0.43,0.10) {\Huge $\Sigma$};
    \end{scope}
\end{tikzpicture}
\end{figure}
On the portion of boundary that $\Sigma$ shares with $\Omega$ the function $u$ is equal to $1$. While on the portion of boundary which $\Sigma$ shares with the environment, $u$ satisfies a Robin condition $$\frac{\partial u}{\partial \nu}+\beta u=0,$$
where $\nu$ is outer normal of the boundary of $\Sigma$.
In fact, %
according to Fourier law, which holds inside $\Sigma$, the heat flux per unit time is $-Du$, and therefore the flux across the surface element at the boundary is $-\frac{\partial u}{\partial \nu}$. On the other hand, according to convection law, the heat flux per unit time across the surface element is equal to $\beta u$ ($\beta$ times the jump of $u$). Continuity of the heat flux enforces to equalize these expressions and the Robin b.c. naturally arises.

In summary, $u$ is a continuous function in $\Omega\cup\Sigma$ which solves
\[\begin{cases}\Delta u=0 &\text{in }\Sigma,\\[5 pt]
 u=1 &\text{in }\Omega,\\[5 pt]
\dfrac{\partial u}{\partial \nu}+\beta u=0 &\text{on }\partial\Sigma\setminus\partial\Omega.
\end{cases}\]

We can also characterize the function $u$ as the minimizer of the energy functional
\[\mathcal{E}(v)=\int_{\Sigma} \abs{\nabla v}^2\,dx+\beta\int_{\partial(\Omega\cup\Sigma)} v^2\,d\Hn,\]
among all functions $v\in H^1(\Omega\cup\Sigma)$ such that $v\equiv 1$ in $\Omega$.

For a given $\Omega$, our ultimate goal would be to find the shape of $\Sigma$ which minimizes $$\int_{\partial(\Omega\cup\Sigma)} u\, d\Hn$$ among all $\Sigma$ of prescribed measure.

So far, such a problem seems to be out of reach, so we restricted our analysis to the case where the layer of the insulating material and the conductivity of the insulator are both very small. %

More precisely, let $\Omega\subset\R^n$ be a smooth bounded, open set, and let $h\colon \partial\Omega\to\R$ be a non-negative function. Denoting by $\nu$ the exterior unit normal to the boundary of $\Omega$,  we define
\[\Sigma_\eps=\Set{\sigma+t\nu(\sigma)|\sigma\in\partial\Omega,\,0<t<\eps h(\sigma)}\]
and we denote by $\Omega_\eps=\overline{\Omega}\cup\Sigma_\eps$. 
Here, the non negative parameter $\eps$ is meant to be small and we want to investigate the limit as $\eps$ vanishes. But in order to have this limit non trivial, we also assume that the conductivity of the insulator is small, namely, the heat flux inside $\Sigma_\eps$ is $-\eps Du$. In practice, we are assuming that we use a "small quantity" of a "very good" insulator.

All at once we can consider the minimization of the following energy functional
\[\mathcal{F}_\eps(v,h)=\eps\int_{\Sigma_\eps} \abs{\nabla v}^2\,dx+\beta\int_{\partial\Omega_\eps} v^2\,d\Hn,\]
where $v\in H^1(\Omega_\eps)$, with $v=1$ in $\Omega$. Here, the small parameter $\eps$ in front of the first integral is encoding the fact that the conductivity of the insulator is small. %
For given $h$, a minimum $u_{\eps,h}$ of
\begin{equation}\label{problema0}\min\Set{\mathcal{F}_\eps(v,h)|v\in H^1(\Omega_\eps),\,v=1\,\text{in }\Omega}\end{equation}
solves the boundary value problem:
\[\begin{cases}\Delta u_{\eps,h}=0 &\text{in }\Sigma_\eps,\\[5 pt]
 u_{\eps,h}=1 &\text{in }\Omega,\\[5 pt]
\eps\dfrac{\partial u_{\eps,h}}{\partial \nu_\eps}+\beta u_{\eps,h}=0 &\text{on }\partial\Omega_\eps\setminus\partial\Omega,
\end{cases}\]
where $\nu_\eps$ is the exterior unit normal to the boundary of $\Omega_\eps$.

Similar problems have been studied before in the context of thermal insulation in \cite{BCF}, \cite{friedman}, \cite{acrbibuttazzo}, and more recently in \cite{bubuni} and \cite{depiniscatro}. The limit has been performed in several ways. In our case, we are going to use $\Gamma$-convergence. But in order to extract as much information as possible about the problem we are going to perform a first order expansion in $\eps$ \cite{anzellotti1993asymptotic}, which, to our knowledge, has never been exploited in this context. \medskip 

The volume of insulator we displace is $\eps m$, for some $m>0$, and we define the volume constraint by defining the space
\begin{equation}
\label{eq: Hm}
\mathcal{H}_m=\mathcal{H}_m(\partial\Omega)=\Set{h\in L^1(\partial\Omega)| \begin{aligned} &\int_{\partial\Omega}h\,d\Hn\le m \\[3 pt]&\,h\ge0\end{aligned}}.
\end{equation}
Our problem reduces to finding the best configuration of insulating material surrounding $\Omega$, that is
\begin{equation}
\label{problema}
    \min\Set{\mathcal{F}_\eps (v,h)| \begin{aligned} &v\in H^1(\Omega_\eps), \\ &v=1\,\text{in 
    }\Omega, \\ & h\in\mathcal{H}_m\end{aligned}}.
\end{equation}\medskip

Following argument similar to those used in \cite{depiniscatro}, %
it can be proved that, for any fixed Lipschitz function $h:\partial\Omega \to [0,+\infty)$, as $\eps\to 0^+$, the functional $\mathcal{F}_\eps(\cdot,h)$ $\Gamma$-converges, in the strong $L^2(\R^n)$ topology, to the functional
\[\mathcal{F}_0(h)=\beta\int_{\partial\Omega} \dfrac{1}{1+\beta h}\,d\Hn.\]
Then, in view of the convexity of the functional with respect to $h$,
\[
\min\Set{\mathcal{F}_0(h) | h\in\mathcal{H}_m}
\]
is achieved by the constant $h=m/P(\Omega)$, where $P(\Omega)$ denotes the perimeter of $\Omega$.

Displacing the insulator uniformly around the boundary is somehow the trivial solution, the one suggested by common sense, and very common when insulating buildings. However, it is mathematically not satisfactory at all. As we expect that portions of boundary with higher (mean) curvature are less convenient to insulate with respect to those with lower curvature. Such an idea is strongly suggested by the radial cases (see for example \cite[Proposition 5.1]{dellapietra}). %

But such a kind of evidence is lost when performing the $\Gamma$-limit and therefore we decided to push our analysis a bit further.
Let 
\[K_0=\Set{v\in L^2(\R^n)|\, v=1\,\text{in } \Omega},\] 
our main result is a first-order asymptotic development by $\Gamma$-convergence (see \autoref{def:asymptotic}) for the functional $\mathcal{F}_\eps$. We denote by $H$ the mean curvature of $\Omega$ (see \autoref{def:H}) and we prove the following
\begin{teor}
\label{teorema1} 
Let $\Omega\subset\R^n$ be a bounded, open set with $C^3$ boundary, and  fix a $C^2$ function $h\colon\partial\Omega\to (0,+\infty)$. Then the functional
\[
\delta\mathcal{F}_\eps(\cdot, h)=\dfrac{\mathcal{F}_\eps(\cdot,h)-\mathcal{F}_0(h)}{\eps}
\]
$\Gamma$-converges, in the strong $L^2(\R^n)$ topology, as $\eps\to 0^+$, to
\[
\mathcal{F}^{(1)}(v,h)=\begin{cases}\displaystyle \beta\int_{\partial\Omega} \frac{Hh(2+\beta h)}{2(1+\beta h)^2}\,d\Hn &\text{if }v\in K_0, \\[10 pt]
+\infty &\text{if }v\in L^2(\R^n)\setminus K_0.\end{cases} \]
\end{teor}
\medskip

The paper is planned as follows. In 
\autoref{section: Gamma} we prove \autoref{teorema1}. Thereafter, in \autoref{shapeopt} we fix $\Omega$ and we deal with the minimum problem
\[\inf\Set{\mathcal{F}_0(h)+\eps \mathcal{F}^{(1)}(h)| h\in\mathcal{H}_m},\]
where $\mathcal{F}^{(1)}(h)=\mathcal{F}^{(1)}(\chi_\Omega,h)$. As we already mentioned, the problem above, is a first-order approximation of the problem \eqref{problema} with respect to $\eps>0$. Indeed we have that (see \autoref{oss: approx}) 
\[\mathcal{F}_\eps(u_\eps,h)=\mathcal{F}_0(h)+\eps \mathcal{F}^{(1)}(h)+R(\Omega,h,\eps),\]
where $u_\eps$ is the minimizer to \eqref{problema0}, and
\[\lim_{\eps\to0^+} \dfrac{R(\Omega,h,\eps)}{\eps}=0.\]
In particular, we will prove that, as the intuition suggests, if $\eps$ is small enough then the optimal configuration for the insulating layer concentrates close to the points of $\partial\Omega$ where the mean curvature is relatively small.
Finally, in \autoref{section:rmk} we discuss the behaviour of the functional $\mathcal{F}_0+\eps \mathcal{F}^{(1)}$ under various geometrical constraints (volume, perimeter, quermassintegral) on the set $\Omega$.

\section{Notation and tools}
\subsection{$\Gamma$-convergence}
In this section, we recall some basic properties of the $\Gamma$-convergence and the asymptotic development by $\Gamma$-convergence. We refer for instance to \cite{dalmaso} and \cite{anzellotti1993asymptotic} for the following notions.
\begin{defi}
Let $X$ be a metric space and, for any $\eps>0$, let us consider the functionals $\mathcal{F}_\eps,\mathcal{F}_0 : X\to \R\cup\set{+\infty}$. We will say that \emph{$\mathcal{F}_\eps$ $\Gamma$-converges, with respect to the strong topology in $X$, as $\eps\to 0^+$} to $\mathcal{F}_0$ if for every $x\in X$ the following conditions hold:
\begin{itemize}
 \item{for every sequence $\set{x_\eps}\subset X$ converging to $x$,
 \[\liminf_{\eps\to0^+}\mathcal{F}_\eps(x_\eps)\geq \mathcal{F}_0(x);\]}
 \item{there exists a sequence $\set{x_\eps}\subset X$ converging to $x$ such that \[\limsup_{\eps\to0^+}\mathcal{F}_\eps(x_\eps)\le \mathcal{F}_0(x).\]}
\end{itemize}
\end{defi}
In particular, from the definition, if $\mathcal{F}_\eps$ $\Gamma$-converges to $\mathcal{F}_0$, for every $x\in X$ there exists a recovery sequence $\set{x_\eps}\subset X$, converging to $x$, such that
\[\lim_{\eps\to0^+} \mathcal{F}_\eps(x_\eps)=\mathcal{F}_0(x).\]
We have the following
\begin{prop}\label{prop:convminimi}
Let $X$ be a metric space and, for any $\eps>0$, let us consider the functionals $\mathcal{F}_\eps,\mathcal{F}_0 : X\to\R\cup\set{+\infty}$ such that $\mathcal{F}_\eps$ $\Gamma$-converges, with respect to the strong topology in $X$ as $\eps\to 0^+$ to $\mathcal{F}_0$. Let $\set{x_\eps}$ be a sequence in $X$ such that \[\mathcal{F}_\eps(x_\eps)=\min_X \mathcal{F}_\eps.\]
If there exists $\overline{x}\in X$ such that $x_\eps$ converges to $\overline{x}$, then
\[\mathcal{F}_0(\overline{x})=\min_X \mathcal{F}_0=\lim_{\eps\to0^+}\min_X \mathcal{F}_\eps.\]
\end{prop}

Let
\[
m_0=\inf_X \mathcal{F}_0,
\]
and, for every $x\in X$, let
\[
\delta\mathcal{F}_\eps(x)=\frac{\mathcal{F}_\eps(x)-m_0}{\eps}
\]
\begin{defi}\label{def:asymptotic}
If there exists a functional $\mathcal{F}^{(1)}\colon X\to \R\cup\set{+\infty}$ such that $\delta\mathcal{F}_\eps$ $\Gamma$-converges, with respect to the strong topology in $X$, as $\eps\to 0^+$ to $\mathcal{F}^{(1)}$, we say that $\mathcal{F}^{(1)}$ is the \emph{first-order asymptotic development by $\Gamma$-convergence for the functional $\mathcal{F}_\eps$}.
\end{defi}

Let
\[\mathcal{U}_0=\Set{x\in X|\mathcal{F}_0(x)=m_0},\]
the interest in the previous definition is justified by the following
\begin{oss}
\label{oss: approx}
Let $\set{x_\eps}$ be a sequence in $X$ such that \[\mathcal{F}_\eps(x_\eps)=\min_X \mathcal{F}_\eps,\]
and assume that there exists $\overline{x}\in X$ such that $x_\eps$ converges to $\overline{x}$; then by \autoref{prop:convminimi} we have that $\overline{x}\in\mathcal{U}_0$ and then
\[\mathcal{F}^{(1)}(\overline{x})=\min_X \mathcal{F}^{(1)}=\lim_{\eps\to0^+}\dfrac{\mathcal{F}_\eps(x_\eps) - m_0}{\eps}.\]
In particular, we have
\[\mathcal{F}_\eps(x_\eps)=m_0+\eps \mathcal{F}^{(1)}(\bar{x})+o(\eps).\]
\end{oss}

\subsection{Calculus on hypersurfaces}
We refer to \cite{maggi} for the notions in this section. Let $\Omega\subset\R^n$ be a bounded open set with $C^1$ boundary and let $\nu$ be the outer unit normal to its boundary. For every $\sigma\in\partial\Omega$ let $\tau=\{\tau_1(\sigma),\dots,\tau_{n-1}(\sigma)\}$ be an orthonormal basis orthogonal to $\nu(\sigma)$, namely a basis for the tangent plane at $\partial\Omega$ in $\sigma$.

\begin{defi}[Tangential gradient]
Let $U\subseteq\R^n$ be an open set containing $\partial\Omega$ and let $\phi=(\phi_1,\dots,\phi_n):U\to\R^n$ be a $C^1$ function. We define the tangential gradient of $\phi$ as the matrix-valued function $D_\tau \phi\colon \partial\Omega\to\R^{n\times(n-1)}$ such that
\[(D_\tau \phi)_{i,j}=\nabla \phi_i\cdot \tau_j,\]
where $i=1,\dots,n$ and  $j=1,\dots,n-1$.
\end{defi}

\begin{defi}[Tangential Jacobian]
Let $U\subseteq\R^n$ be an open set containing $\partial\Omega$ and let $\phi: U\to\R^n$ be a $C^1$ function. We define the tangential Jacobian of $\phi$ as 
\[J_\tau \phi =\sqrt{\det\big((D_\tau \phi)^T (D_\tau \phi)\big)\,}\,.\]
\end{defi}

\begin{teor}[Area formula on surfaces]\label{teor:area}
Let $U\subseteq\R^n$ be an open set containing $\partial\Omega$, let $\phi: U\to\R^n$ be a $C^1$ function,    and let $g:\R^n\to\R$ be a positive Borel function. We have that
\[\int_{\partial\Omega} g(\phi(\sigma))\, J_\tau \phi\,d\Hn=\int_{\phi(\partial\Omega)} g(\sigma)\,d\Hn.\]
\end{teor}

\begin{defi}[Tangential divergence]
Let $U\subseteq\R^n$ be an open set containing $\partial\Omega$ and let $\phi: U\to\R^n$ be a $C^1$ function. We define the tangential divergence of $\phi$ as 
\[
\divv_\tau \phi = \sum_{j=1}^{n-1} (D \phi \,\tau_j)\cdot\tau_j
\]
\end{defi}
\begin{defi}[Mean Curvature]\label{def:H}
Let $\Omega$ be a bounded open set with $C^2$ boundary and let $\nu$ be the outer unit normal to its boundary. Let $U\subseteq\R^n$ be an open set containing $\partial\Omega$ and let $X$ be a $C^1(U)$ extension of $\nu$. For any $\sigma\in\partial\Omega$ we define the \emph{mean curvature of $\partial\Omega$} as
\[H(x)=\divv_\tau \nu.\]
\end{defi}

\begin{oss}\label{ossimp}
Let $U\subseteq\R^n$ be an open set containing $\partial\Omega$, let $X\colon U\subseteq\R^n\to\R^n$ be a $C^1$ function, and let 
$\phi(x)=x+tX(x)$. By direct computations, we have that
\[J_\tau \phi(\sigma)=1+t\divv_\tau X(\sigma)+t^2R(t,\sigma)\]
where the remainder $R$ is a bounded function. In particular, if $\Omega$ has $C^2$ boundary and $X$ is an extension of $\nu$, we have   
\[J_\tau \phi(\sigma)=1+t H(\sigma)+ t^2R(t,\sigma).\]
\end{oss}
\begin{defi}[Set of Finite perimeter]
Let $E\subseteq\R^n$ be a measurable set. We define the \emph{perimeter} of $E$ as
\[
P(E)=\sup\Set{\int_E \divv\varphi\,d\Ln | \begin{aligned}
\varphi\in &\:C_c(\Omega,\R^n) \\ &\abs{\varphi}\le 1
\end{aligned}}.
\]
If $P(E)<+\infty$ we say that $E$ is a \emph{set of finite perimeter}.
\end{defi}
\begin{defi}[Generalized mean curvature] Let $\Omega\subset\R^n$ be a set of finite perimeter, and let $p\in[1,+\infty]$. We say that $\Omega$ has \emph{generalized mean curvature in $L^p$} if there exists $H_\Omega\in L^p(\partial\Omega)$ such that
\[
\int_{\partial\Omega}\divv_\tau F\,d\Hn =\int_{\partial \Omega}H_\Omega F\cdot \nu \,d\Hn,
\]
for any $F\in C^\infty_C(A;\R^n)$ with $A$ open set containing $\Omega$.
\end{defi}

\begin{teor}[Coarea formula]\label{coarea}
Let $f\colon\R^n\to\R$ be a Lipschitz function, let $g\colon\R^n\to\R$ be an $L^1(\R^n)$ function and let $U\subset\R^n$ be an open set, then
\[\int_U g(x)\abs{\nabla f(x)}\,dx=\int_{\R} \int_{U\cap\Set{f=t}} g(y)\,d\Hn(y)\,dt.\]
\end{teor}

\section{The $\Gamma$-limit}
\subsection{Setting of the problem}
\label{setting}
Let $\Omega\subset\R^n$ be a bounded, open set with $C^{1,1}$ boundary, and  fix a positive Lispchitz function $h\colon\partial\Omega\to\R$. We recall that
\[\Sigma_\eps=\Set{\sigma+t\nu(\sigma)|\sigma\in\partial\Omega,\,0<t<\eps h(\sigma)},\]
and 
\[\Omega_\eps=\overline{\Omega}\cup\Sigma_\eps.\]
Our assumptions on $\partial\Omega$ ensure that there exists $\eps_0=\eps_0(\Omega,h)$ such that, if $0<\eps\le\eps_0$, the map
\[(\sigma,t)\longmapsto\sigma+t\nu(\sigma)\]
is invertible, that is, for every $x\in\Sigma_\eps$ there exist unique $\sigma(x)\in\partial\Omega$ and $t(x)$, with $0\le t(x)\le \eps h(\sigma(x))$, such that \[x=\sigma(x)+t(x)\nu(\sigma(x)).\]
Therefore, we can extend $h$ and $\nu$ on $\Sigma_\eps$ as $h(x)=h(\sigma(x))$, and $\nu(x)=\nu(\sigma(x))$ respectively. Moreover, for every $x\in \R^n\setminus\Omega$, let 
\[
d(x)=d(x,\partial\Omega)=\inf_{y\in\partial\Omega}\,\abs{x-y}
\]
be the distance from  $\partial\Omega$, then we have that $d(x)=t(x)$ for every $x\in\Sigma_\eps$.
\begin{oss}
By coarea formula (\autoref{coarea}), the area formula on surfaces (\autoref{teor:area}) and \autoref{ossimp}, we have that if $f:\Omega_\eps\to\R$ is a positive Borel function, then
\begin{equation}\label{eq:intsigma0} \int_{\Sigma_\eps}f(x)\,dx=\int_{\partial\Omega}\int_0^{\eps h(\sigma)} f(\sigma+t\nu)\left(1+\eps R_1(\sigma,t,\eps)\right)\,dt\,d\Hn\end{equation}
and 
\begin{equation}\label{eq:intdesigma0} \int_{\partial\Omega_\eps}f(\sigma)\,d\Hn=\int_{\partial\Omega} f(\sigma+\eps h\nu)\left(1+\eps R_2(\sigma,\eps)\right)\,d\Hn,\end{equation}
where the remainder terms $R_1$ and $R_2$ are bounded functions, then there exists $Q_0>0$ such that $\abs{R_1},\abs{R_2}\le Q_0$. Moreover, if $\Omega$ is a bounded open set with $C^3$ boundary, then we have
\begin{equation}\label{eq:intsigma} \int_{\Sigma_\eps}f(x)\,dx=\int_{\partial\Omega}\int_0^{\eps h(\sigma)} f(\sigma+t\nu)\left(1+tH(\sigma)+\eps^2 R_3(\sigma,t,\eps)\right)\,dt\,d\Hn\end{equation}
and 
\begin{equation}\label{eq:intdesigma} \int_{\partial\Omega_\eps}f(\sigma)\,d\Hn=\int_{\partial\Omega} f(\sigma+\eps h\nu)\left(1+\eps h(\sigma)H(\sigma)+\eps^2 R_4(\sigma,\eps)\right)\,d\Hn,\end{equation}
where the remainder terms $R_3$ and $R_4$ are bounded functions, then there exists $Q>0$ such that $\abs{R_3},\abs{R_4}\le Q$.
\end{oss}
\medskip 

Let 
\begin{equation}
\label{def: Keps}
K_\eps=\Set{v\in H^1(\Omega_\eps)|\, v=1\,\text{in }\Omega},
\end{equation}
and 
\begin{equation}
\label{def: K0}
K_0=\Set{v\in L^2(\R^n)|\, v=1\,\text{in }\Omega},
\end{equation}
and consider the functional
\[
\mathcal{F}_\eps(v,h)=
\begin{cases}
    \displaystyle \eps\int_{\Sigma_\eps} \abs{\nabla v}^2\,dx+\beta\int_{\partial\Omega_\eps} v^2\,d\Hn &\text{if }v\in K_\eps,\\[10 pt]
    +\infty &\text{if } v\in L^2(\R^n)\setminus K_\eps.
\end{cases}
\]
denoting by
\[
\mathcal{F}_0(v,h)=\begin{cases}\displaystyle \beta\int_{\partial\Omega} \dfrac{1}{1+\beta h}\,d\Hn &\text{if }v\in K_0, \\[10 pt]
+\infty &\text{if }v\in L^2(\R^n)\setminus K_0,
\end{cases}\]
following the approach of \cite{depiniscatro}, we have the following
\begin{prop}
Let $\Omega\subset\R^n$ be a bounded, open set with $C^{1,1}$ boundary, and  fix a Lispchitz function $h\colon\partial\Omega\to(0,+\infty)$. Then $\mathcal{F}_\eps(\cdot,h)$ $\Gamma$-converges, as $\eps\to0^+$,  in the strong $L^2(\R^n)$ topology, to $\mathcal{F}_0(\cdot,h)$.
\end{prop}
\begin{proof}
We start by proving the $\Gamma$-liminf inequality: Let $v\in L^2(\R^n)$ and let $v_\eps\in L^2(\R^n)$ such that $v_\eps$ converges to $v$ in $L^2(\R^n)$ as $\eps\to0^+$. Up to passing to a sub-sequence, we can assume that 
\[\liminf_{\eps\to0^+}\mathcal{F}_\eps(v_\eps,h)=\lim_{\eps\to0^+}\mathcal{F}_\eps(v_\eps,h),\]
moreover, we can assume that such a limit is finite and that $v_\eps\in K_\eps$. Therefore we have that $v\in K_0$ and, by \eqref{eq:intsigma0}, \eqref{eq:intdesigma0} we have that 
\begin{equation}\label{eq:1}\int_{\Sigma_\eps}\abs{\nabla v_\eps}^2\,dx\ge\int_{\partial\Omega} \int_0^{\eps h(\sigma)}\abs{\nabla v_\eps(\sigma+t\nu)}^2 (1-\eps Q_0)\,dt\,d\Hn\end{equation}
and
\begin{equation}\label{eq:2}\int_{\partial\Omega_\eps} v_\eps^2\,d\Hn\ge\int_{\partial\Omega} v_\eps^2(\sigma+\eps h(\sigma)\nu) (1-\eps Q_0)\,d\Hn.\end{equation}
On the other hand, we have that, for $\Hn$-almost every $\sigma\in\partial\Omega$,
\[\begin{split}\int_0^{\eps h(\sigma)}\abs{\nabla v_\eps(\sigma+t\nu)}^2 \,dt\ge&\dfrac{1}{\eps h(\sigma)}\left(\int_0^{\eps h(\sigma)}\abs{\nabla v_\eps(\sigma+t\nu)}\,dt\right)^2\\[10 pt]
\ge&\dfrac{(v_\eps(\sigma+\eps h\nu)-1)^2}{\eps h(\sigma)},
\end{split}\]
then, by Young's inequality, we have that, for every $\lambda>0$ and for $\Hn$-almost every $\sigma\in\partial\Omega$, 
\begin{equation}\label{intyoung}\int_0^{\eps h(\sigma)}\abs{\nabla v_\eps(\sigma+t\nu)}^2 \,dt\ge\dfrac{(1-\lambda) v_\eps(\sigma+\eps h\nu)^2}{\eps h(\sigma)}+\dfrac{1}{\eps h(\sigma)}\left(1-\dfrac{1}{\lambda}\right).\end{equation}
Putting together \eqref{eq:1}, \eqref{eq:2} and \eqref{intyoung} we finally have
\[\mathcal{F}_\eps(v_\eps,h)\ge\int_{\partial\Omega}\left(\left(\dfrac{1-\lambda}{h}+\beta\right)v_\eps(\sigma+\eps h\nu)^2+\dfrac{1}{h}\left(1-\dfrac{1}{\lambda}\right)\right)\,d\Hn-\eps Q_0 R_\eps(\eps,v_\eps),\]
where, if $\eps$ is sufficiently small, using again \eqref{eq:intsigma0} and \eqref{eq:intdesigma0}, we have
\[R_\eps(\eps,v_\eps)\le 2\mathcal{F}_\eps(v_\eps,h).\]
Finally, letting $\lambda=\lambda(\sigma)=1+\beta h(\sigma)$, and passing to the limit as $\eps\to0^+$ we have that
\[\liminf_{\eps\to0^+}\mathcal{F}_\eps(v_\eps,h)\ge\beta\int_{\partial\Omega}\dfrac{1}{1+\beta h}\,d\Hn=\mathcal{F}_0(v,h)\]
and the $\Gamma$-liminf inequality is proved. 

\medskip
$\Gamma$-limsup inequality: Let $v\in L^2(\R^n)$, if $v\notin K_0$ the $\Gamma$-limsup inequality is trivial, therefore let $v\in K_0$.
Let
\[
v_\eps(x)=\begin{dcases}1 &\text{if }x\in\Omega, \\[5 pt]
1-\dfrac{\beta d(x)}{\eps(1+\beta h(x))} &\text{if }x\in\Sigma_\eps,\\[5 pt]
v(x) &\text{if }x\notin\Omega_\eps,
\end{dcases}\]
where we recall that, if $x=\sigma+t\nu(\sigma)$, then $h(x)=h(\sigma)$. Trivially $v_\eps$ converges to $v$ in $L^2(\R^n)$ and $v_\eps\in K_\eps$. For every $x\in\Sigma_\eps$,
\[\nabla v_\eps(x)=-\dfrac{\beta \nabla d(x)}{\eps(1+\beta h(x))}+\dfrac{\beta^2 d(x)\nabla h(x)}{\eps(1+\beta h(x))^2}.\]
Recalling that $0\le d\le \eps h$, $\nabla d=\nu$ and $\nabla h\cdot \nu=0$, we have
\[\abs{\nabla v_\eps}^2=\dfrac{\beta^2 }{\eps^2(1+\beta h)^2}+\dfrac{\beta^4 d^2 \abs{\nabla h}^2}{\eps^2(1+\beta h)^4}\le \dfrac{\beta^2 }{\eps^2(1+\beta h)^2}+ \dfrac{\beta^4 h^2 \abs{\nabla h}^2}{(1+\beta h)^4},\]
where the second term is bounded since $h$ is Lipschitz. Hence, substituting $\eps h \tau=t$ in \eqref{eq:intsigma0}, we get 
\[\begin{split}\eps\int_{\Sigma_\eps}\abs{\nabla v_\eps}^2\,dx &\le \dfrac{\beta^2 }{\eps}\int_{\Sigma_\eps}\dfrac{1}{(1+\beta h)^2}\,dx+\eps C\abs{\Sigma_\eps}\\[10 pt]
&\le\int_{\partial\Omega}\int_0^1 \dfrac{\beta^2 h}{(1+\beta h)^2}(1+\eps Q_0)\,d\tau\,d\Hn+\eps C\abs{\Sigma_\eps}\\[10 pt]
&=\int_{\partial\Omega}\dfrac{\beta^2 h}{(1+\beta h)^2}\,d\Hn+o(\eps).
\end{split}\]
On the other hand, for every $\sigma\in\partial\Omega$,
\[v_\eps(\sigma+\eps h(\sigma)\nu(\sigma))=\dfrac{1}{1+\beta h(\sigma)},\]
from which we get
\[\beta\int_{\partial\Omega_\eps} v_\eps^2\,d\Hn\le \int_{\partial\Omega}\dfrac{\beta}{(1+\beta h)^2}(1+\eps Q_0)\,d\Hn=\int_{\partial\Omega}\dfrac{\beta}{(1+\beta h)^2}\,d\Hn+o(\eps).\]
Hence we have
\[\mathcal{F}_\eps(v_\eps,h)\le\beta \int_{\partial\Omega}\dfrac{1}{1+\beta h}\,d\Hn+o(\eps),\]
so that
\[\limsup_{\eps\to0^+}\mathcal{F}_\eps(v_\eps,h)\le\beta \int_{\partial\Omega}\dfrac{1}{1+\beta h}\,d\Hn\]
and the $\Gamma$-limsup inequality is proved.
\end{proof}
In the following, for simplicity, we will denote by
\[
\mathcal{F}_0(h)=\beta\int_{\partial\Omega}\frac{1}{1+\beta h}\,d\Hn.
\]
It can be deduced from the more general results in \cite{depiniscatro} that the minimum in the class of functions $h$ with a given mass of such functional is achieved when $h$ is constant. We include a direct proof of this statement in the following
\begin{prop}\label{prop: minimosenzaeps}
Let $\Omega$ be a bounded open set with Lipschitz boundary, let $P(\Omega)=P$, and let $m>0$. Then the problem
\begin{equation}\label{eq:problema00}\min\Set{\mathcal{F}_0(h)| h\in\mathcal{H}_m}\end{equation}
admits 
\[h_0=\dfrac{m}{P}\]
as the unique solution.
\end{prop}
\begin{proof}
Let $h\in\mathcal{H}_m$. By Holder's inequality, we have that
\[\begin{split} P&=\int_{\partial\Omega}\,d\Hn\le\left(\int_{\partial\Omega}\dfrac{1}{1+\beta h}\,d\Hn\right)^{1/2}\left(\int_{\partial\Omega}(1+\beta h)\,d\Hn\right)^{1/2}\\[10 pt]
&\le\left(\int_{\partial\Omega}\dfrac{1}{1+\beta h}\,d\Hn\right)^{1/2}\left(P+\beta m\right)^{1/2},
\end{split}\]
so that
\[\mathcal{F}_0(h)\ge\dfrac{\beta P^2}{P+\beta m}=\mathcal{F}_0(h_0).\]
Finally, the uniqueness of the solution is given by the strict convexity of the function
\[x\longmapsto \dfrac{1}{1+\beta x}\]
for $x\ge0$.
\end{proof}\medskip 

Let $H$ be the mean curvature of $\Omega$, we aim to show that 
\[
\delta\mathcal{F}_\eps(\cdot, h) =\dfrac{\mathcal{F}_\eps(\cdot,h)-\mathcal{F}_0(h)}{\eps}
\]
$\Gamma$-converges, in the strong $L^2(\R^n)$ topology, to 
\[\mathcal{F}^{(1)}(v,h)=\begin{cases}\displaystyle \beta\int_{\partial\Omega} \dfrac{ Hh(2+\beta h)}{2(1+\beta h)^2}\,d\Hn &\text{if }v\in K_0, \\[10 pt]
+\infty &\text{if }v\in L^2(\R^n)\setminus K_0,  \end{cases}\]
where $K_0$ is the set defined in \eqref{def: K0}.

\subsection{Proof of \autoref{teorema1}}
\label{section: Gamma}
Let $\Omega$ be a bounded, open set with $C^3$ boundary, and fix a positive $C^2$ function $h\colon\partial\Omega\to\R$. In this section, we study the $\Gamma$-convergence of the family of functionals
\begin{equation}\label{Gesp}\delta\mathcal{F}_\eps(v)=\dfrac{\mathcal{F}_\eps(\cdot,h)-\mathcal{F}_0(h)}{\eps},\end{equation}
and we prove \autoref{teorema1}. In the following we  consider the functions $h, H\colon\partial\Omega\to\R$ extended on the set $\Sigma_\eps$ as $h(\sigma+t\nu)=h(\sigma)$ and $H(\sigma+t\nu)=H(\sigma)$. \medskip

For every $\eps>0$ let $u_\eps\in K_\eps$ be the minimizer to $\mathcal{F}_\eps$, where $K_\eps$ is defined in \eqref{def: Keps}. By the assumptions on $\Omega$ and $h$, we have that $u_\eps$ is a $C^2(\Sigma_\eps)$ function and it is a solution to
  \begin{equation}\label{elu}\begin{cases}
 -\Delta u_\eps = 0 &\text{in }\Sigma_\eps,\\[5 pt]
 u_\eps=1 &\text{on }\partial\Omega,\\[5 pt]
 \eps\dfrac{\partial u_\eps}{\partial\nu_\eps}+\beta u_\eps= 0 &\text{on }\partial\Omega_\eps.
 \end{cases}\end{equation}
Let $\alpha\in(0,\alpha_0)$, where
\[
\alpha_0=1-\max_{\sigma\in\partial\Omega}\frac{\beta h(\sigma)}{1+\beta h(\sigma)},
\]
and let $\eps>0$, then on $\Sigma_\eps$ we define
 \[
w_{\eps,\alpha}(x)=\begin{dcases}
1-\left(\frac{d(x)}{\eps h(x)}\right)^{1-\alpha}\frac{\beta h(x)}{(1-\alpha)(1+\beta h(x))} &\text{if }H(x)\ge 0,\\[7 pt]
1-\left(\frac{d(x)}{\eps h(x)}\right)^{1+\alpha}\frac{\beta h(x)}{(1+\alpha)(1+\beta h(x))} &\text{if }H(x)< 0. 
\end{dcases}
\]
Then we have that $w_{\eps,\alpha}>0$ and the following holds
 \begin{prop}\label{prop:subsol}
 For every $\alpha\in(0,\alpha_0)$ there exists $\eps_\alpha>0$ such that if $0<\eps<\eps_\alpha$, then 
 \[
 H(x) u_\eps(x) \ge H(x) w_{\eps,\alpha}(x) \qquad \text{for }x\in\Sigma_\eps.
 \]
 \end{prop}
 \begin{proof}
 Fix $\alpha\in(0,\alpha_0)$. For simplicity, we denote by
 \[
 v_\gamma:=1-\left(\frac{d(x)}{\eps h(x)}\right)^{\gamma}\frac{\beta h(x)}{\gamma(1+\beta h(x))},
 \]
 and we aim to show that there exists an $\eps_\alpha>0$ such that for any $0<\eps<\eps_\alpha$, we have that $v_{1-\alpha}$ is a subsolution to \eqref{elu}, while $v_{1+\alpha}$ is a supersolution to the same problem. Namely,
 \begin{equation}
 \label{subsol}
     \begin{cases}
     -\Delta v_{1-\alpha}\le 0 &\text{in }\Sigma_\eps,\\[5 pt]
     v_{1-\alpha}=1 &\text{on }\partial\Omega,\\[5 pt]
     \eps\dfrac{\partial v_{1-\alpha}}{\partial\nu_\eps}+\beta v_{1-\alpha}\le 0 &\text{on }\partial\Omega_\eps,
     \end{cases} \qquad \qquad
     \begin{cases}
     -\Delta v_{1+\alpha}\ge 0 &\text{in }\Sigma_\eps,\\[5 pt]
     v_{1+\alpha}=1 &\text{on }\partial\Omega,\\[5 pt]
     \eps\dfrac{\partial v_{1+\alpha}}{\partial\nu_\eps}+\beta v_{1+\alpha}\ge 0 &\text{on }\partial\Omega_\eps.
     \end{cases}
 \end{equation}
 In the following, we will always assume that $\eps<1$. Let us recall that
\[
\Omega_\eps=\Set{x\in\R^n | \dfrac{d(x)}{h(x)}\le\eps},\qquad \partial\Omega_\eps=\Set{x\in\R^n | d(x)-\eps h(x)=0}.
\]
By standard computations we get
\[
\nabla\left(\frac{d}{h}\right)=\frac{\nabla d}{h}-\frac{d\nabla h}{h^2}, \qquad \abs*{\nabla\left(\frac{d}{h}\right)}=\frac{1}{h}\sqrt{1+\left(\frac{d}{h}\right)^2\abs{\nabla h}^2}.
\]
Then, recalling that $\nabla d=\nu$ and that $\nabla h\cdot \nu=0$, the normal $\nu_\eps$ to the set $\Omega_\eps$ is given by
\[
\nu_\eps=\frac{1}{\sqrt{1+\eps^2\abs{\nabla h}^2}}(\nu-\eps\nabla h).
\]
By direct computations, for any $\gamma\in(0,2)\setminus\{1\}$ we have
\begin{equation}
\label{eq: wlap}
\Delta v_{\gamma} =-\frac{\beta h}{\gamma \eps^\gamma (1+\beta h)}\Delta\!\left[\left(\frac{d}{h}\right)^\gamma\right]-\frac{2}{\gamma\eps^\gamma}\nabla\!\left[\left(\frac{d}{h}\right)^\gamma\right]\cdot\nabla\!\left[\frac{\beta h}{1+\beta h}\right]-\frac{1}{\gamma\eps^\gamma}\left(\frac{d}{h}\right)^\gamma\Delta \!\left[\frac{\beta h}{1+\beta h}\right]. 
\end{equation}
We then compute
\begin{equation}
\label{eq: gradients}
\nabla\!\left[\left(\frac{d}{h}\right)^\gamma\right]=\gamma\left(\frac{d}{h}\right)^{\gamma-1}\left(\frac{\nu}{h}-\frac{d\nabla h}{h^2}\right), \qquad \nabla\!\left[\frac{\beta h}{1+\beta h}\right]=\frac{\beta\nabla h}{(1+\beta h)^2},
\end{equation}
from which we get
\begin{equation}
\label{eq: scalarproductw}
\nabla\!\left[\left(\frac{d}{h}\right)^\gamma\right]\cdot\nabla\!\left[\frac{\beta h}{1+\beta h}\right]=-\gamma\left(\frac{d}{h}\right)^\gamma\frac{\beta\abs{\nabla h}^2}{h(1+\beta h)^2}.
\end{equation}
In addition, we have
\begin{equation}
\label{eq: lapdh}
\begin{split}
\Delta\!\left[\left(\frac{d}{h}\right)^\gamma\right]&=\gamma(\gamma-1)\left(\frac{d}{h}\right)^{\gamma-2}\abs*{\nabla\left(\frac{d}{h}\right)}^2+\gamma\left(\frac{d}{h}\right)^{\gamma-1}\left(\frac{\Delta d}{h}-\frac{d\Delta h}{h^2}+2\frac{d\abs{\nabla h}^2}{h^3}\right) \\[10 pt]
&=\gamma\left(\frac{d}{h}\right)^{\gamma-2}\left(-\frac{1-\gamma}{h^2}-\left(\frac{d}{h}\right)^2\frac{(1-\gamma)\abs{\nabla h}^2}{h^2}+\frac{d}{h}\frac{\Delta d}{h}-\left(\frac{d}{h}\right)^2\frac{\Delta h}{h}+\left(\frac{d}{h}\right)^2\frac{2\abs{\nabla h}^2}{h^2}\right),
\end{split}
\end{equation}
so that, by \eqref{eq: wlap}, \eqref{eq: lapdh}, and \eqref{eq: scalarproductw}, we get 
\begin{equation}
\label{eq: lapwestim}
\left(\frac{d(x)}{h(x)}\right)^{2-\gamma}\eps^\gamma\Delta v_{\gamma}(x)=\frac{\beta(1-\gamma)}{h(x)(1+\beta h(x))}+R_1(x,\eps,\gamma),
\end{equation}
where $R_1(x,\eps,\gamma)$ is a suitable remainder term. Since $d\le \eps h$, 
\[
0<\inf_{\Sigma_\eps} h\le \sup_{\Sigma_\eps} h<+\infty,
\]
and $\abs{\nabla h}$, $\Delta h$, $\Delta d$ are bounded, then there exist $C_\gamma,\eps_0>0$ such that
\begin{equation}
\label{eq: estimrem}
\abs{ R_1(x,\eps,\gamma) }\le C_\gamma\eps
\end{equation}
for any $\eps<\eps_0$.
Thus, using \eqref{eq: estimrem} in \eqref{eq: lapwestim} we have that there exists $\eps_\alpha>0$ such that if $0<\eps<\eps_\alpha$, then 
\begin{equation}
\label{eq: finallapest}
-\Delta v_{1-\alpha}<0, \qquad -\Delta v_{1+\alpha} >0.
\end{equation}

\medskip
On the other hand, for every $x\in\partial\Omega_\eps$, since $d(x)=\eps h(x)$ and \eqref{eq: gradients} hold true, we get
\[
\nabla v_{\gamma}(x) =-\frac{\beta}{\eps(1+\beta h(x))}(\nu(x)-\eps\nabla h(x))-\frac{\beta \nabla h(x)}{\gamma(1+\beta h(x))^2},
\]
which yields
\begin{equation}\label{normder}
\begin{split}
\frac{\partial v_{\gamma}}{\partial \nu_\eps}(x)=&-\frac{\beta\sqrt{1+\eps^2\abs{\nabla h}^2}}{\eps(1+\beta h)}-\frac{\beta}{\gamma}\nu_\eps\cdot\frac{\nabla h}{(1+\beta h)^2} \\[7 pt]
=&-\frac{\beta\sqrt{1+\eps^2\abs{\nabla h}^2}}{\eps(1+\beta h)}+\frac{\beta \eps\abs{\nabla h}^2}{\gamma(1+\beta h)^2\sqrt{1+\eps^2\abs{\nabla h}^2}},
\end{split}
\end{equation}
while
\begin{equation}\label{boundaryv}
v_{\gamma}(x)=1-\frac{\beta h(x)}{\gamma(1+\beta h(x))}.
\end{equation}
Hence, we get by \eqref{normder} and \eqref{boundaryv}
\[
\eps\frac{\partial v_{\gamma}}{\partial \nu_\eps}+\beta v_{\gamma} =-(1-\gamma)\frac{\beta^2 h}{\gamma(1+\beta h)}+R_2(\sigma,\eps,\gamma)\quad\text{on }\partial\Omega_\eps,
\]
where, as before, up to choosing a smaller $\eps_0$,
\[
\abs{R_2(\sigma,\eps,\gamma)}\le C_\gamma \eps.
\]
Again, for small enough $\eps$, on $\partial\Omega_\eps$ we get
\begin{equation}\label{robinw}
\eps\frac{\partial v_{1-\alpha}}{\partial \nu_\eps}+\beta v_{1-\alpha}<0,\quad\qquad \eps\frac{\partial v_{1+\alpha}}{\partial \nu_\eps}+\beta v_{1+\alpha}>0.
\end{equation}

Finally, joining \eqref{eq: finallapest} and \eqref{robinw}, by standard comparison results for elliptic operators the proposition is proved.
 \end{proof}
 
 We can now prove \autoref{teorema1}.
 \begin{proof}[Proof of \autoref{teorema1}]
 We start by proving the $\Gamma$-liminf inequality: 
without loss of generality, we can prove the inequality for the sequence of minimizers $u_\eps$. Here we recall the definitions of $\mathcal{F}_\eps$ and $\mathcal{F}_0$, omitting the dependence on $h$.
 \begin{equation}
 \label{eq: defFeps}
 \mathcal{F}_\eps(u_\eps)=\eps\int_{\Sigma_\eps} \abs{\nabla u_\eps}^2\,dx+\beta\int_{\partial\Omega_\eps} u_\eps^2\,d\Hn,
 \end{equation}
 \begin{equation}
  \label{eq: defF0}
 \mathcal{F}_0=\beta\int_{\partial\Omega} \dfrac{1}{1+\beta h}\,d\Hn.
 \end{equation} 
 By \eqref{eq:intsigma} and \eqref{eq:intdesigma} we have \begin{equation}\label{cov1}\int_{\Sigma_\eps}\abs{\nabla u_\eps}^2\,dx\ge\int_{\partial\Omega}\int_0^{\eps h(\sigma)} \abs{\nabla u_\eps(\sigma+t\nu)}^2\left(1+tH(\sigma)-\eps^2 Q\right)\,dt\,d\Hn\end{equation}
and
\begin{equation}\label{cov2}\dfrac{\beta}{\eps}\int_{\partial\Omega_\eps} u_\eps^2\,\Hn\ge\dfrac{\beta}{\eps}\int_{\partial\Omega} u_\eps^2(\sigma+\eps h(\sigma)\nu(\sigma))\left(1+\eps h(\sigma)H(\sigma)-\eps^2 Q\right)\,d\Hn.\end{equation}
For $\eps$ sufficiently small, for every $\sigma\in\partial\Omega$, and $0<t<\eps h(\sigma)$, we have that $1+tH(\sigma)>0$, so that, using Holder's inequality and integrating by parts,
\[\begin{split}\int_0^{\eps h(\sigma)} \abs{\nabla u_\eps(\sigma+t\nu)}^2\left(1+tH(\sigma)\right)\,dt&\ge\dfrac{1}{\eps h}\left(\int_0^{\eps h(\sigma)} \abs{\nabla u_\eps(\sigma+t\nu)}\sqrt{1+tH}\,dt\right)^2\\[10pt]
&\ge\dfrac{1}{\eps h}\left(\int_0^{\eps h(\sigma)}\dfrac{d}{dt}( u_\eps(\sigma+t\nu))\sqrt{1+tH}\,dt\right)^2\\[10 pt]
&\ge\frac{1}{\eps h}\left(u_\eps(\sigma+\eps h\nu)\sqrt{1+\eps h H}-\left(1+\displaystyle\int_0^{\eps h(\sigma)}\dfrac{H u_\eps(\sigma+t\nu)}{2\sqrt{1+tH}}\,dt\right)\right)^2.
\end{split}\]
Up to choosing a smaller $\eps$, we can apply Young's inequality, having that for every $\lambda>0$
\begin{equation}
\label{eq: lowergrad}
\begin{split}\int_0^{\eps h(\sigma)} \abs{\nabla u_\eps^2(\sigma+t\nu)}\left(1+tH(\sigma)\right)\,dt\ge& \dfrac{(1-\lambda)(1+\eps h H)u_\eps(\sigma+\eps h\nu)^2}{\eps h}\\[10 pt]&+\dfrac{1}{\eps h}\left(1-\dfrac{1}{\lambda}\right)\left(1+\int_0^{\eps h(\sigma)}\dfrac{H u_\eps(\sigma+t\nu)}{2\sqrt{1+tH}}\,dt\right)^2. \end{split}
\end{equation}
We then have, joining \eqref{eq: defFeps},\eqref{cov1}, \eqref{eq: lowergrad}, \eqref{cov2}, and \eqref{eq: defF0},
\begin{equation}\label{4asterischi}\begin{split}\delta\mathcal{F}_\eps(u_\eps)=\dfrac{\mathcal{F}_\eps(u_\eps)-\mathcal{F}_0}{\eps}\ge&\int_{\partial\Omega}\dfrac{1}{\eps h(\sigma)}\left((1-\lambda)(1+\eps hH)+\beta h(1+\eps h H)\right)u_\eps^2(\sigma+\eps h\nu)\,d\Hn\\[10 pt]
&+\int_{\partial\Omega}\dfrac{1}{\eps h}\left(\left(1-\dfrac{1}{\lambda}\right)\left(1+\int_0^{\eps h}\dfrac{H u_\eps(\sigma+t\nu)}{2\sqrt{1+tH}}\,dt\right)^2-\dfrac{\beta h}{1+\beta h}\right)\,d\Hn\\[10 pt]
&-Q\eps R(\eps,u_\eps)\end{split}\end{equation}
where, if $\eps$ is small enough, 
\[\begin{split}R(\eps,u_\eps)&=\eps\int_{\partial\Omega}\int_0^{\eps h(\sigma)} \abs{\nabla u_\eps(\sigma+t\nu)}^2\,d\Hn+\beta\int_{\partial\Omega} u_\eps(\sigma+\eps h(\sigma)\nu(\sigma))^2\,d\Hn\\[10 pt]
&\le 2\mathcal{F}_\eps(u_\eps).
\end{split}\]
Letting $\lambda=\lambda(\sigma)=1+\beta h(\sigma)$ in \eqref{4asterischi}, and using the inequality $(1+x)^2\ge1+2x$,
\[
    \delta\mathcal{F}_\eps(u_\eps)\ge\int_{\partial\Omega}\dfrac{\beta h H}{\eps (1+\beta h)}\int_0^{\eps }\dfrac{u_\eps(\sigma+th\nu)}{\sqrt{1+thH}}\,dt\,d\Hn+O(\eps).
\]
Moreover, for every $t\in(0,\eps)$ we have that $(1+thH)^{-1/2}=1+O(\eps)$, so that
\begin{equation}\label{stimaliminf}\delta\mathcal{F}_\eps(u_\eps)\ge\beta \int_{\partial\Omega}\dfrac{ h H}{(1+\beta h)}\fint_0^{\eps }u_\eps(\sigma+th\nu)\,dt\,d\Hn+O(\eps).\end{equation}
Finally, let $\alpha\in(0,1)$ and let 
\[\gamma=\gamma(\sigma)=\begin{cases} 1-\alpha &\text{if }H(\sigma)\ge0\\[5 pt] 1+\alpha &\text{if }H(\sigma)<0.
\end{cases}
\] Let us recall that
\[
w_{\eps,\alpha}(\sigma+th(\sigma)\nu(\sigma))=1-t^\gamma\frac{\beta h(\sigma)}{\eps^\gamma \gamma(1+\beta h(\sigma))}.
\]
By \autoref{prop:subsol} we have that for every $0<\eps<\eps_\alpha$
\[
\begin{split}
\delta\mathcal{F}_\eps(u_\eps)&\ge\beta \int_{\partial\Omega}\dfrac{ h H}{(1+\beta h)}\fint_0^{\eps }w_{\eps,\alpha}(\sigma+th\nu)\,dt\,d\Hn+O(\eps)\\[10 pt]
&=\int_{\partial\Omega}\left(1-\frac{\beta h}{(1+\beta h)\gamma(\gamma+1)}\right)\frac{\beta h H}{1+\beta h}\,d\Hn+O(\eps),
\end{split}
\]
so that
\[\liminf_{\eps\to0^+}\delta\mathcal{F}_\eps(u_\eps)\ge\int_{\partial\Omega}\left(1-\frac{\beta h}{(1+\beta h)\gamma(\gamma+1)}\right)\frac{\beta h H}{1+\beta h}\,d\Hn.\]
Letting $\alpha$ go to $0$, we have that $\gamma$ tends to $1$, and
\[\liminf_{\eps\to0^+}\delta\mathcal{F}_\eps(u_\eps)\ge\beta\int_{\partial\Omega}\dfrac{hH(2+\beta h)}{2(1+\beta h)^2}\,d\Hn,\]
and the $\Gamma$-Liminf is proved.\medskip

We now prove the $\Gamma$-limsup inequality:

Let
\[
\varphi_\eps(x)=\begin{dcases}1 &\text{if }x\in\Omega, \\[5 pt]
1-\dfrac{\beta d(x)}{\eps(1+\beta h(x))}-\frac{\beta d(x)^2 H(x)}{2\eps(1+\beta h(x))^2} &\text{if }x\in\Sigma_\eps,\\[5 pt]
0 &\text{if }x\in\R^n\setminus\Omega,
\end{dcases}\]
where we recall that if $x=\sigma+t\nu(\sigma)$, then $h(x)=h(\sigma)$ and $H(x)=H(\sigma)$. 
We have that $\varphi_\eps\in H^1(\Omega)$ and $\varphi_\eps$ converges in $L^2(\R^n)$, to the characteristic function of $\Omega$. Computing the gradient of $\varphi_\eps$, for any $x\in\Sigma_\eps$,
\[
\nabla \varphi_\eps(x)=-\frac{\beta\nu(x)}{\eps(1+\beta h(x))}-\frac{\beta d(x)H(x)\nu(x)}{\eps(1+\beta h(x))^2}+R_\eps(x),
\]
where $R_\eps$ is a remainder term which is bounded, uniformly in $\eps$, since $d\le\eps h$, and $h,\nabla h, H, \nabla H$ are bounded. Moreover, $\nu\cdot R_\eps=0$, since both $h$ and $H$ are defined in such a way that $\nabla h$ and $\nabla H$ are orthogonal to $\nu$. Therefore, for sufficiently small $\eps$,
\[
\begin{split}
\abs{\nabla\varphi_\eps}^2&\le\frac{\beta^2}{\eps^2}\underbracket{\left(\dfrac{1}{(1+\beta h)^2}+2\dfrac{Hd}{(1+\beta h)^3}\right)}_{\ge 0}+C,
\end{split}
\]
where we used again the boundedness of $h$ and $H$, and the fact that $d\le \eps h$. Hence, substituting $\eps h\tau = t$ in \eqref{eq:intsigma}, and noticing that $d(\sigma+\eps \tau h(\sigma)\nu(\sigma))=\eps h \tau$, we get
\begin{equation}\label{stimagrad}
\begin{split}
    \eps\int_{\Sigma_\eps} \abs{\nabla \varphi_\eps}^2\,dx\le&\dfrac{\beta^2}{\eps}\int_{\Sigma_\eps}\left(\dfrac{1}{(1+\beta h)^2}+2\dfrac{Hd}{(1+\beta h)^3}\right)\,dx+\eps C\abs*{\Sigma_\eps}\\[10 pt]
    \le&\int_{\partial\Omega}\int_0^1\dfrac{\beta^2h}{(1+\beta h)^2}(1+\eps \tau h H+\eps^2 Q)\,d\tau\,d\Hn+\\[10 pt]
    &+\int_{\partial\Omega}\int_0^1\dfrac{2\eps\beta^2 h^2H \tau}{(1+\beta h)^3}(1+\eps \tau h H+\eps^2 Q)\,d\tau\,d\Hn
    +\eps C\abs*{\Sigma_\eps}\\[10 pt]
    =&\int_{\partial\Omega}\dfrac{\beta^2h}{(1+\beta h)^2}\,d\Hn+\eps\int_{\partial\Omega}\left(\dfrac{\beta^2h^2H}{2(1+\beta h)^2}+\dfrac{\beta^2 h^2H }{(1+\beta h)^3}\right)\,d\Hn+O(\eps^2).
\end{split}
\end{equation}
On the other hand, for every $\sigma\in\partial\Omega$,
\[
\varphi_\eps(\sigma+\eps h(\sigma)\nu(\sigma))=\frac{1}{1+\beta h(\sigma)}-\frac{\eps\beta h(\sigma)^2 H}{2(1+\beta h(\sigma))^2},
\]
from which we get
\begin{equation}\label{phi2}\begin{split}
\beta\int_{\partial{\Omega_\eps}} \varphi_\eps^2\,d\Hn\le\int_{\partial\Omega}\dfrac{\beta}{(1+\beta h)^2}(1+\eps h H)\,d\Hn
-\int_{\partial\Omega}\dfrac{\eps\beta^2 h^2 H}{(1+\beta h)^3}\,d\Hn + O(\eps^2).
\end{split}
\end{equation}
Finally, joining \eqref{eq: defFeps}, \eqref{stimagrad}, \eqref{phi2}, and \eqref{eq: defF0} we have
\[\begin{split}\delta\mathcal{F}_\eps(\varphi_\eps)=\dfrac{\mathcal{F}_\eps(u_\eps)-\mathcal{F}_0}{\eps}&\le\beta\int_{\partial\Omega}\left(\dfrac{\beta^2h^2H}{2(1+\beta h)^2}+\dfrac{\beta hH}{(1+\beta h)^2}\right)\,d\Hn +O(\eps)\\[10 pt]
&=\beta\int_{\partial\Omega}\dfrac{hH(2+\beta h)}{2(1+\beta h)^2}\,d\Hn +O(\eps)
\end{split}\]
so that
\[\limsup_{\eps\to0^+}\delta\mathcal{F}_\eps(\varphi_\eps)\le\beta\int_{\partial\Omega}\dfrac{h H(2+\beta h)}{2(1+\beta h)^2}\,d\Hn\]
and the $\Gamma$-limsup inequality is proved.
 \end{proof}
 
\section{Properties of the first order development}\label{shapeopt}
Let $\Omega$ be a bounded, open set with $C^{1,1}$ boundary. Consider the functional
\[\mathcal{G}_\eps(\Omega,h)=\beta\int_{\partial\Omega}\left(\dfrac{1}{1+\beta h}+\eps H\dfrac{h(2+\beta h)}{2(1+\beta h)^2}\right)\,d\Hn.\]
In the following, we drop the dependence on the set $\Omega$ and we write $\mathcal{G}_\eps(h)$ in place of $\mathcal{G}_\eps(\Omega,h)$.
For every $m>0$, we will consider the problem
\begin{equation}\label{problemm}\inf\Set{\mathcal{G}_\eps(h)|h\in \mathcal{H}_m},\end{equation}
where the set $\mathcal{H}_m$ is defined in \eqref{eq: Hm}.
\begin{oss}
Assume that a non-zero continuous solution $\mu\in L^1(\partial\Omega)$ to problem \eqref{problemm} exists, so that the set $U=\set{\mu>0}$ is open. Then for every $\psi\in C_c^\infty(U)$ with zero mean, and for every $\eta\in\R$ sufficiently small, we can consider the variation $\mu+\eta\psi$ which leads to the Euler-Lagrange equation
\[\int_{\partial\Omega}\left(-\dfrac{\beta}{(1+\beta \mu)^2}+\dfrac{\eps H}{1+\beta \mu}-\eps H \dfrac{\beta \mu(2+\beta \mu)}{(1+\beta \mu)^3}\right)\psi\,d\Hn=0.\]
The previous equation yields
\[\dfrac{c}{\beta}(1+\beta \mu)^3-(1+\beta\mu)+\dfrac{\eps H}{\beta} =0,\]
for some constant $c\in\R$.
\end{oss}

Let $\eps>0$, in the following we will assume that 
\begin{equation}\label{hpH}\sup_{\partial\Omega}\dfrac{\eps H}{\beta}\le\dfrac{2}{3}.\end{equation}
Let 
\[H_0=\inf_{\partial\Omega} H\]
and 
\[k_0=1-\dfrac{\eps H_0}{\beta}.\]
For every $k\in(0,k_0)$, let \[\Gamma_k=\Set{\sigma\in\partial\Omega|\dfrac{\eps H(\sigma)}{\beta}< 1-k}\] and consider 
\[P_k(y,\sigma)=ky^3-y+\dfrac{\eps H(\sigma)}{\beta}.\]

\begin{prop}\label{yk} 
Let \eqref{hpH} hold true. Then, for every $k\in(0,k_0)$, and $\sigma\in \Gamma_k$, in the interval $(1,+\infty)$ there exists a unique $y_k(\sigma)$ such that 
\begin{equation}\label{implicit} P_k(y_k(\sigma),\sigma)=0,\end{equation}
and there exists $z_k> 1$ such that
\begin{equation}
\label{ykestiates}
    \max\left\{\dfrac{1}{\sqrt{3k}},1\right\}\le y_k(\sigma)\le z_k.
\end{equation}
In particular, we have
\begin{equation}
\label{eq: limitzk}
\lim_{k\to k_0^-}z_k=1.
\end{equation}
Moreover, for every $k_1<k_2$ and $\sigma\in\Gamma_{k_2}$ we have that 
\begin{equation}\label{monotonicity} y_{k_2}(\sigma)<y_{k_1}(\sigma).\end{equation}
\end{prop}
\begin{proof}
For any fixed $\sigma\in\Gamma_k$ we have that $P_k(1,\sigma)<0$, and in addition, for $k\ge 1/3$, the polynomial $P_k(y,\sigma)$ is strictly increasing in $y\ge1$, while for $k<1/3$ we have that
\begin{align*}
&\frac{\partial}{\partial y}P_k(y,\sigma)<0 \qquad \text{if }y\in \left[1,\frac{1}{\sqrt{3k}}\right),\\[5 pt]
&\frac{\partial}{\partial y}P_k(y,\sigma)>0 \qquad \text{if }y\in \left(\frac{1}{\sqrt{3k}},+\infty\right).
\end{align*}
Therefore, in the interval $(1,+\infty)$ there exists a unique zero $y_k(\sigma)$ of the polynomial $P_k(\cdot,\sigma)$, and 
\[
y_k(\sigma)\ge\frac{1}{\sqrt{3k}}.
\]
Notice in addition that for every $y> 1$, we have that $y<y_k(\sigma)$ if and only if $P_k(y,\sigma)<0$. 
Hence, if we choose $z_k$ to be the unique real number in $(1,+\infty)$ such that
\begin{equation}
\label{eq: zkdef}
kz_k^3-z_k+\frac{\eps H_0}{\beta}=0,
\end{equation}
then
\[
P_k(z_k,\sigma)=\frac{\eps H(\sigma)}{\beta}-\frac{\eps H_0}{\beta}\ge 0,
\]
and we have that \eqref{ykestiates} holds. 

\medskip
We now prove \eqref{eq: limitzk}. We first observe that $z_k$ is decreasing in $k$: let $k_1<k_2$, so that 
\[k_1 z_{k_2}^3-z_{k_2}+\dfrac{\eps H_0}{\beta}<k_2 z_{k_2}^3-z_{k_2}+\dfrac{\eps H_0}{\beta}=0=k_1 z_{k_1}^3-z_{k_1}+\dfrac{\eps H_0}{\beta},\]
which ensures
\[z_{k_2}<z_{k_1},\] 
since the polynomial $k_1y^3-y+\eps H_0/\beta$ is strictly increasing in $[1/\sqrt{3k_1},+\infty)$. We now have that there exists
\[
z=\lim_{k\to k_0^-}z_k,
\]
and, passing to the limit in \eqref{eq: zkdef} and recalling that by definition $\beta k_0=\eps H_0$, we get that $z$ solves the equation 
\[
k_0(z^3-1)-z+1=0.
\]
From \eqref{hpH}, we have that $k_0>1/3$, so that $z=1$ is the unique solution in $[1,+\infty)$ to the previous equation, proving \eqref{eq: limitzk}.

\medskip
Finally, in order to prove \eqref{monotonicity}, let $k_1<k_2$ and $\sigma\in\Gamma_{k_1}\cap\Gamma_{k_2}=\Gamma_{k_2}$, then \eqref{implicit} ensures that
\[
P_{k_1}(y_{k_2},\sigma)< P_{k_2}(y_{k_2},\sigma)=0,
\]
from which
\[
y_{k_2}(\sigma)<y_{k_1}(\sigma).
\]
\end{proof}
Let $k\in(0,k_0)$, and let $y_k$ be as in \autoref{yk}, we define
\[\mu_k(\sigma)=\begin{cases}
\dfrac{1}{\beta}(y_k(\sigma)-1) & \text{if }\sigma\in\Gamma_k.\\[10 pt]
0 &\text{if }\sigma\in\partial\Omega\setminus\Gamma_k,
\end{cases}\]
Notice that, by \eqref{monotonicity}, $\mu_k$ is decreasing in $k$. 

We have the following
\begin{prop}
Let \eqref{hpH} hold true. Then, for every $m>0$ ,there exists a unique $k=k_m\in(0,k_0)$ such that
\[
\int_{\partial\Omega} \mu_k\,d\Hn=m.
\]

\end{prop}
\begin{proof}
We first prove that the function
\[
M(k)=\int_{\partial\Omega}\mu_k\,d\Hn
\]
is continuous.  Fix $k\in(0,k_0)$ and let $\delta>0$, then
\begin{equation}
\label{eq: limit+}
\beta(M(k)-M(k+\delta))=\int_{\Gamma_{k+\delta}}(y_{k}-y_{k+\delta})\,d\Hn+\int_{\Set{1-k-\delta\le \frac{\eps H}{\beta}< 1-{k}}}(y_{{k}}-1)\,d\Hn. 
\end{equation}
By definition, we have that for every $\sigma\in\partial\Omega$
\[
\lim_{\delta\to 0^+}\chi_{\Gamma_{k+\delta}}(\sigma)=\chi_{\Gamma_k}(\sigma).
\]
Let $\sigma\in \Gamma_{k}$, then the function $y_{k+\delta}(\sigma)$ is defined for small enough $\delta<\delta_\sigma$,  and by the implicit function theorem and the regularity of $P_k(y,\sigma)$, we get
\[
\lim_{\delta\to 0^+}y_{k+\delta}(\sigma)=y_k(\sigma).
\]
 Therefore, by \eqref{ykestiates}, the monotonicity of $z_k$, and the dominated convergence theorem we have
\begin{equation}
\label{eq: firstintegral+}
\lim_{\delta\to 0^+}\int_{\Gamma_{k+\delta}}(y_{k}-y_{k+\delta})\,d\Hn=0.
\end{equation}
On the other hand, for every $\sigma\in\partial\Omega$,
\[
\lim_{\delta\to 0^+}\chi_{\Set{1-k-\delta\le \frac{\eps H}{\beta}< 1-{k}}}(\sigma)=0,
\]
which entails
\begin{equation}
\label{eq: secondintegral+}
    \lim_{\delta\to 0^+}\int_{\Set{1-k-\delta\le \frac{\eps H}{\beta}< 1-{k}}}(y_{{k}}-1)\,d\Hn=0.
\end{equation}
Joining \eqref{eq: limit+}, \eqref{eq: firstintegral+}, and \eqref{eq: secondintegral+}, we get
\[
\lim_{\delta\to 0^+} M(k)-M(k+\delta)=0.
\]
We now fix $k\in(0,k_0)$, $\delta>0$, and we compute
\begin{equation}
    \label{eq: limit-}
    \beta(M(k-\delta)-M(k))=\int_{\Gamma_{k}}(y_{k-\delta}-y_{k})\,d\Hn+\int_{\Set{1-k\le \frac{\eps H}{\beta}< 1-{k}+\delta}}(y_{{k-\delta}}-1)\,d\Hn. 
\end{equation}
As in the previous case, by the implicit function theorem, for every $\sigma\in\Gamma_k$,
\[\lim_{\delta\to0^+}y_{k-\delta}(\sigma)=y_k(\sigma).\]
By the dominated convergence theorem,
\begin{equation}
    \label{eq: firstintegral-}
    \lim_{\delta\to 0^+}\int_{\Gamma_k}(y_{k-\delta}-y_k)\,d\Hn=0.
\end{equation}
On the other hand, we have that for every $\sigma\in\partial\Omega$
\[\lim_{\delta\to 0^+}\chi_{\Set{1-k\le \frac{\eps H}{\beta}< 1-{k}+\delta}}(\sigma)=\chi_{\Set{\frac{\eps H}{\beta}=1-k}}(\sigma),\]
and, for every $\sigma$ such that ${\eps H(\sigma)}=\beta(1-k)$, we may use the monotonicity of $y_{k-\delta}$ and then we pass to the limit in \eqref{implicit}, having that
\[\lim_{\delta\to0^+}y_{k-\delta}(\sigma)=1,\]
which entails
\begin{equation}
\label{eq: secondintegral-}
\lim_{\delta\to0^+}\int_{\Set{1-k\le \frac{\eps H}{\beta}< 1-{k}+\delta}}(y_{{k-\delta}}-1)\,d\Hn=0.
\end{equation}
Joining \eqref{eq: limit-}, \eqref{eq: firstintegral-}, and \eqref{eq: secondintegral-}, we get
\[
\lim_{\delta\to 0^+}M(k-\delta)-M(k)=0,
\]
thus concluding the proof of the continuity of $M$.

Finally, by monotonicity, and by \eqref{ykestiates}, and \eqref{hpH}, we have that
\[
\lim_{k\to0^+}M(k)=+\infty,\quad\lim_{k\to k_0^-}M(k)=0.
\]
and the proposition is proved. 
\end{proof}
\begin{teor}\label{teor:exist}
Let \eqref{hpH} hold true. Then, for every $m>0$, the function $\mu_{k_m}$ is the unique minimizer to problem \eqref{problemm}.
\end{teor}
\begin{proof}
Let $h\colon \partial\Omega\to R$ with $h\ge0$ and 
\[\int_{\partial\Omega}h\,d\Hn\le m.\]
for every $t\in[0,1]$ consider
\[h_t=\mu_k+t(h-\mu_k)\]
and
\[g(t)=\mathcal{G}_\eps(h_t).\]
We claim that $g(t)$ is increasing in $t$.
By explicit computation, we have that
\[g'(t)=\beta\int_{\partial\Omega}\dfrac{\eps H-\beta(1+\beta h_t)}{(1+\beta h_t)^3}(h-\mu_k)\,d\Hn.\]
From \eqref{hpH} we have that, for every $\sigma\in \partial\Omega$ the function
\[x\mapsto \dfrac{\eps H-\beta(1+\beta x)}{(1+\beta x)^3}\]
is increasing on $[0,+\infty)$, so that
\[\begin{split} g'(t)&\ge \beta\int_{\partial\Omega}\dfrac{\eps H-\beta(1+\beta \mu_k)}{(1+\beta \mu_k)^3}(h-\mu_k)\,d\Hn\\[10 pt]
&=\beta\left(\int_{\partial\Omega\setminus\Gamma_k} (\eps H-\beta)h\,d\Hn+\int_{\Gamma_k}\dfrac{\eps H-\beta(1+\beta \mu_k)}{(1+\beta \mu_k)^3}(h-\mu_k)\,d\Hn\right).
\end{split}\]
By \eqref{implicit} we have that on $\Gamma_k$
\[\dfrac{\eps H-\beta(1+\beta \mu_k)}{(1+\beta \mu_k)^3}=-\beta k\]
while on $\partial\Omega\setminus\Gamma_k$
\[\eps H-\beta\ge-\beta k\]
so that
\[g'(t)\ge\beta^2 k\int_{\partial\Omega}(\mu_k-h)\,d\Hn\ge0,\]
and the claim is proven.  In particular, we have that
\[\mathcal{G}_\eps(\mu_k)=g(0)\le g(1)=\mathcal{G}_\eps(h)\]
that is, $\mu_k$ is a minimizer for problem \eqref{problemm}. Finally, by \eqref{hpH}, we have that, for every $\sigma\in\partial\Omega$,  the function
\[x\in[0,+\infty)\mapsto\dfrac{1}{1+\beta x}+\eps H(\sigma)\dfrac{x(2+\beta x)}{2(1+\beta x)^2}\]
is strictly convex, thus problem \eqref{problemm} admits a unique minimizer.
\end{proof}
\begin{oss}
Notice that the optimal configuration $\mu_{k_m}$
 concentrates where the mean curvature is smaller: for simplicity, let us write $k=k_m$, and let us take $\sigma_1,\sigma_2\in \Gamma_{k}$ such that
\[
H(\sigma_1)<H(\sigma_2).
\]
Noticing that
\[
P_k(y_k(\sigma_1),\sigma_1)=0=P_k(y_k(\sigma_2),\sigma_2),
\]
we get
\[
ky_k(\sigma_1)^3-y_k(\sigma_1)>ky_k(\sigma_2)^3-y_k(\sigma_2).
\]
By \eqref{ykestiates}, we can use the monotonicity of the function $y\mapsto ky^3-y$ on $[1/\sqrt{3k},+\infty)$, getting
\[
y_k(\sigma_1)>y_k(\sigma_2),
\]
so that $\mu_k(\sigma_1)>\mu_k(\sigma_2)$. 
\end{oss}

\begin{oss}\label{oss:bigH}
Notice that, if for every $\sigma\in\partial\Omega$
\begin{equation}\label{hpH2}\dfrac{\eps H(\sigma)}{\beta}\ge2,\end{equation} 
then the optimal configuration is given by $\mu\equiv0$. Indeed if \eqref{hpH2} holds, for every $\sigma\in \partial\Omega$ the function
\[x\in[0,+\infty)\mapsto \dfrac{1}{1+\beta x}+\dfrac{\eps H(\sigma)}{\beta}\dfrac{\beta x(2+\beta x)}{2(1+\beta x)^2}\]
reaches its minimum for $x=0$.

\end{oss}
\subsection{Minimization with perimeter constraint}
Fix $\eps>0$. For every $P>0$ and $m>0$ we define, 
\[
\mathcal{K}_P=\Set{\Omega\subset\R^n | \begin{aligned}
&\Omega\text{ open and bounded with $C^{1,1}$ boundary}\\
&H_\Omega\ge 0\\
&P(\Omega)=P
\end{aligned}},
\]
where $H_\Omega$ is the mean curvature of $\Omega$. Let $\Omega\in\mathcal{K}_P$, and let $h\in\mathcal{H}_m(\partial\Omega)$, where $\mathcal{H}_m$ is the set defined in \eqref{eq: Hm}, then we study the functional
\[\mathcal{G}_\eps(\Omega,h)=\beta\int_{\partial\Omega}\left(\dfrac{1}{1+\beta h}+\eps H_\Omega\dfrac{h(2+\beta h)}{2(1+\beta h)^2}\right)\,d\Hn.\]
We will now consider the problem
\begin{equation}\label{Pproblem}\inf\Set{\mathcal{G}_\eps(\Omega,h)|(\Omega,h)\in\mathcal{K}_p\times\mathcal{H}_m(\partial\Omega)}.\end{equation}

\begin{defi}[Cookie Shape]
\label{defi: cookie}
For any $r,R>0$ we define the \emph{cookie shape}
\[
C_{r,R}=\Set{(x',x_n) | -f_{r,R}(x')\le x_n\le f_{r,R}(x')},
\]
where
\[
f_{r,R}(x')=\begin{cases}
r &\abs{x'}\le R, \\
\sqrt{r^2-\left(\abs{x'}-R\right)^2} &R<\abs{x'}\le R+r.
\end{cases}
\]
\end{defi}
\begin{oss}
For every $r,R>0$, We have that $C_{r,R}$ is a convex set with $C^{1,1}$ boundary and \[H=H_{C_{r,R}}\ge \frac{1}{r}\chi_{\set{H>0}}.\]
Moreover, 
\begin{equation}
    \label{eq: cookieper}
    P(C_{r,R})=2\omega_{n-1}\left(R^{n-1}+(n-1)r\int_0^1\frac{\left(r\rho+R\right)^{n-2}}{\sqrt{1-\rho^2}}\,d\rho\right).
\end{equation}
We observe that the function $P(C_{r,R})$ is increasing in $r$ and $R$.
\end{oss}

\begin{teor}
For every $P,m>0$ we have
\[\inf\Set{\mathcal{G}_\eps(\Omega,h)|(\Omega,h)\in\mathcal{K}_p\times\mathcal{H}_m(\partial\Omega)}=\dfrac{\beta P^2}{P+\beta m}\]
and the infimum is asymptotically achieved by
a sequence of thin cookie shapes.
\end{teor}
\begin{proof}
For every $ (\Omega,h)\in\mathcal{K}_p\times\mathcal{H}_m(\partial\Omega)$,
\[\mathcal{G}_\eps(\Omega,h)\ge \beta\int_{\partial\Omega}\dfrac{1}{1+\beta h}\,d\Hn\ge\dfrac{\beta P^2}{P+\beta m}.\]
Let $r_k>0$ be a decreasing sequence with
\[\lim_{k}r_k=0;\]
let $R_k$ be such that, for every $k$,
\[P(C_{r_k,R_k})=P.\]
Then $R_k$ is increasing in $k$ and
\[\lim_{k} R_k=\left(\dfrac{P}{2\omega_{n-1}}\right)^\frac{1}{n-1}.\]
Consider
\[h_k(\sigma)=\begin{cases} \dfrac{m}{2\omega_{n-1}R_k^{n-1}} &\text{if }H(\sigma)=0,\\[5 pt]
0 &\text{if }H(\sigma)>0.
\end{cases}\]
We have that
\[\mathcal{G}_\eps(C_{r_k,R_k},h_k)=\dfrac{2\beta\omega_{n-1}R_k^{n-1}}{1+\frac{m}{2\omega_{n-1}R_k^{n-1}}}+\beta\left(P(C_{r_k,R_k})-2\omega_{n-1}R_k^{n-1}\right).\]
Passing to the limit for $k$ to infinity, we have
\[
\lim_{k}\mathcal{G}_\eps(C_{r_k,R_k},h_k)=\dfrac{\beta P^2}{P+\beta m}=\min_{h\in\mathcal{H}_m}\mathcal{F}_0(h).
\]
\end{proof}

\subsection{Maximization with geometric constraints}\label{section:rmk}
Fix $\eps,m>0$. For every bounded, open set $\Omega\subseteq\R^n$ with $C^{1,1}$ boundary, we let
\[\mathcal{G}_\eps(\Omega)=\inf\Set{\mathcal{G}_\eps(\Omega,h)|\, h\in\mathcal{H}_m},\]
 where $\mathcal{H}_m$ is the set defined in \eqref{eq: Hm}. This section will study the maximization of $\mathcal{G}_\eps(\Omega)$ with fixed quermassintegral. We refer to \cite{schneider,burago} for the following
\begin{defi}[Quermassintegrals]
Let $\Omega\subseteq\R^n$ be a nonempty,
bounded, convex set. We define the \emph{quermassintegrals} as the unique coefficients $W_j(\Omega)$ such that
\[
\abs{\Omega+tB}=\sum_{j=0}^{n}\binom{n}{j}W_j(\Omega)t^j,
\]
where $B$ is the unit ball in $\R^n$, and 
\[
\Omega+tB=\Set{x+ t y | x\in \Omega,\: y\in B}.
\]
\end{defi}
In particular $W_0(\Omega)$ is the measure of $\Omega$ and $W_n(\Omega)=\omega_n$, the measure of the unit ball.
\begin{teor}[Alexandrov-Fenchel Inequality]
\label{alex-fench}
Let $0\le i<j\le n-1$, and let  $\Omega\subseteq\R^n$ be a nonempty,
bounded, convex set, then
\[
\left(\frac{W_i(\Omega)}{\omega_n}\right)^{\frac{1}{n-i}}\le \left(\frac{W_j(\Omega)}{\omega_n}\right)^{\frac{1}{n-j}}.
\]
Moreover, the inequality holds as equality if and only if $\Omega$ is a ball.
\end{teor}

\begin{oss}
    Let $\Omega$ be a bounded, open, convex set with $C^2$ boundary and nonzero Gaussian curvature, then the quermassintegral are related to the principal curvatures of the boundary of $\Omega$. Indeed we have that for every $j=1,\dots,n$
\[W_j=\dfrac{1}{n}\int_{\partial\Omega} H_{j-1}(\sigma)\,d\Hn\]
Here $H_j$ denotes the $j$-th normalized elementary symmetric function of the principal curvatures
of $\partial \Omega$, that is $H_0=1$ and, for every $j=1,\dots n-1$,
\[H_j(\sigma)=\binom{n-1}{j}^{-1}\sum_{1\le i_1\le\dots\le i_j\le n-1} k_{i_1}(\sigma)\cdots k_{i_j}(\sigma),\]
where $k_1(\sigma),\dots,k_{n-1}(\sigma)$  are the principal curvatures at a point $\sigma\in\partial\Omega$. In particular, we have that
\[W_1(\Omega)=\dfrac{1}{n} P(\Omega)\]
and
\[W_2(\Omega)=\dfrac{1}{n(n-1)}\int_{\partial\Omega} H_\Omega \,d\Hn.\]
\end{oss}
In the planar case, we have the following
\begin{prop}
\label{prop: ballmaxn2}
    Let $n=2$, and let $\Omega$ be a bounded, open, simply connected set $\Omega$ with $C^2$ boundary such that either 
\[P(\Omega)\ge 3\pi\dfrac{\eps}{\beta}\]
or
\[P(\Omega)\le \pi\dfrac{\eps}{\beta}.\]
Then
\begin{equation}
\label{eq: maxball}
\mathcal{G}_\eps(\Omega)\le \mathcal{G}_\eps(\Omega^*),
\end{equation}
where $\Omega^*$ is the ball having the same perimeter as $\Omega$.
\end{prop}
\begin{proof}
If
\[P(\Omega)\le \pi\dfrac{\eps}{\beta},\]
then
\[\dfrac{\eps H_{\Omega^*}}{\beta}\ge 2,\]
and, by \autoref{oss:bigH}, we get
\[\mathcal{G}_\eps(\Omega^*)=\mathcal{G}_\eps(\Omega^*,0)=\beta P(\Omega^*).\]
On the other hand,
\[\mathcal{G}_\eps(\Omega)\le \mathcal{G}_\eps(\Omega,0)=\beta P(\Omega)=\beta P(\Omega^*),\]
which gives \eqref{eq: maxball}.
If
\[P(\Omega)\ge 3\pi\dfrac{\eps}{\beta},\]
then
\[\dfrac{\eps H_{\Omega^*}}{\beta}\le \dfrac{2}{3},\]
and, by \autoref{teor:exist}, we get
\[\mathcal{G}_\eps(\Omega^*)=\mathcal{G}_\eps(\Omega^*,m/P(\Omega^*)).\]
On the other hand,
\[\begin{split}\mathcal{G}_\eps(\Omega)&\le \mathcal{G}_\eps(\Omega,m/P(\Omega^*))\\
&=\beta\left(\dfrac{P(\Omega^*)P(\Omega)}{P(\Omega^*)+\beta m}+\dfrac{\eps m(2P(\Omega^*)+\beta m)}{2(P(\Omega^*)+\beta m)^2}\int_{\partial\Omega} H_\Omega\,d\Hn\right).
\end{split}\]
By Gauss-Bonnet theorem we have that
\[\int_{\partial\Omega} H_\Omega\,d\Hn=2\pi=\int_{\partial \Omega^*} H_{\Omega^*}\,d\Hn,\]
so that
\[
    \begin{split}
        \mathcal{G}_\eps(\Omega)&\le \beta\left(\dfrac{P(\Omega^*)^2}{P(\Omega^*)+\beta m}+\dfrac{\eps m(2P(\Omega^*)+\beta m)}{2(P(\Omega^*)+\beta m)^2}\int_{\partial \Omega^*} H_{\Omega^*}\,d\Hn\right)\\[10 pt]
        &=\mathcal{G}_\eps(\Omega^*,m/P(\Omega^*)),
    \end{split}
\]
which is \eqref{eq: maxball}.
\end{proof}

In the general case of possibly higher dimensions, we have
\begin{prop}
\label{prop: ballmaxn}

   Let $n\ge3$, $2\le k\le n-1$,  and let $\Omega$ be a bounded, open, convex set with $C^2$ boundary and nonzero Gaussian curvature such that either 
\begin{equation}
\label{eq: condquermass1}
W_{k}(\Omega)\ge\omega_n\left(\dfrac{3(n-1)\eps}{2\beta}\right)^{n-k}
\end{equation}
or
\begin{equation}
\label{eq: condquermass2}
W_{k}(\Omega)\le\omega_n\left(\dfrac{(n-1)\eps}{2\beta}\right)^{n-k}.
\end{equation}
Then 
\[\mathcal{G}_\eps(\Omega)\le\mathcal{G}_\eps(\Omega^*)\] 
where $\Omega^*$ is the ball such that $W_{k}(\Omega)=W_k(\Omega^*)$.

\end{prop}
\begin{proof}
Since $W_k(\Omega)=W_k(\Omega^*)$, from the Alexandrov-Fenchel inequalities, we have that for every $0\le i\le k$
\[\left(\frac{W_i(\Omega^*)}{\omega_n}\right)^{\frac{1}{n-i}}= \left(\frac{W_k(\Omega^*)}{\omega_n}\right)^{\frac{1}{n-k}}=\left(\frac{W_k(\Omega)}{\omega_n}\right)^{\frac{1}{n-k}}\ge\left(\frac{W_i(\Omega)}{\omega_n}\right)^{\frac{1}{n-i}},\]
that is 
\[W_i(\Omega)\le W_i(\Omega^*).\] 
In particular, we have that
\[P(\Omega)\le P(\Omega^*)\]
and
\[\int_{\partial\Omega} H_\Omega\,d\Hn\le\int_{\partial \Omega^*} H_{\Omega^*}\,d\Hn.\]
Moreover, since 
\[
\left(\frac{W_k(\Omega^*)}{\omega_n}\right)^{\frac{1}{n-k}}= \left(\frac{W_1(\Omega^*)}{\omega_n}\right)^{\frac{1}{n-1}}=\left(\dfrac{P(\Omega^*)}{n\omega_n}\right)^{\frac{1}{n-1}}=\frac{n-1}{H_{\Omega^*}},
\]
the conditions \eqref{eq: condquermass1} and \eqref{eq: condquermass2} read as
\[\dfrac{\eps H_{\Omega^*}}{\beta}\le \dfrac{2}{3},\]
and
\[\dfrac{\eps H_{\Omega^*}}{\beta}\ge 2,\]
respectively. Therefore, the result can be obtained following the proof of \autoref{prop: ballmaxn2}.
\end{proof}
\subsection{Final remarks}
Let $\Omega,\Omega_0\subseteq\R^n$ be bounded open sets
with smooth boundary such that $\overline{\Omega}\subset\overline{\Omega}_0$ and let
\[
\mathcal{I}_\beta(\Omega,\Omega_0)=\inf\Set{\int_{\Omega_0} \abs{\nabla u}^2\,dx+\beta\int_{\partial\Omega_0}u^2\,d\Hn|\begin{aligned}&\, u\in H^1(\Omega_0),\\&\,u\ge1\,\text{ in }\overline{\Omega}\end{aligned}}.
\]
The results in \autoref{prop: ballmaxn2} and \autoref{prop: ballmaxn} are coherent with the ones proved in \cite{dellapietra} for the functional 
\[\mathcal{I}_{\beta,\delta}(\Omega)=\mathcal{I}_\beta(\Omega,\Omega+\delta B),\]
where $\Omega_0=\Omega+\delta B$ is the  Minkowski sum of $\Omega$ and the unit ball $B=B_1(0)\subset\R^n$. Namely, 
in \cite{dellapietra}, the following theorems are proved
\begin{teor}
    Let $\Omega$ be a connected, bounded, open set in $\R^2$ with piecewise $C^1$ boundary. Then
\[\mathcal{I}_{\beta,\delta}(\Omega)\le\mathcal{I}_{\beta,\delta}(\Omega^*),\] 
where $\Omega^*$ is the ball having the same perimeter as $\Omega$.
\end{teor}
\begin{teor}
    Let $n\ge3$ and let $\Omega$ be a bounded, open, convex set in $\R^n$. Then
\[\mathcal{I}_{\beta,\delta}(\Omega)\le\mathcal{I}_{\beta,\delta}(\Omega^*),\] 
where $\Omega^*$ is the ball such that $W_{n-1}(\Omega^*)=W_{n-1}(\Omega)$.
\end{teor}

Finally, in \cite{BNNT}, it is proved the following
\begin{teor}
    The solution to the problem
\[
\min\Set{I_{\beta}(\Omega,\Omega_0)|
\begin{aligned}
&\abs{\Omega}=\omega_n, \\ &\abs{\Omega_0}\le M 
\end{aligned}
}
\]
consists of two concentric balls.
\end{teor}
The previous theorem naturally leads to the following question 
\begin{open}
Prove or disprove that the problem
\[
\min\Set{\mathcal{G}_\eps(\Omega,h)| \begin{aligned} &\abs{\Omega}=\omega_n,\\ & h\in\mathcal{H}_m\end{aligned}}
\]
admits the couple $(B_1,h^*)$ as a solution, where $h^*$ is a constant function.
\end{open}

\printbibliography[heading=bibintoc]
\Addresses
\end{document}